\documentclass[11pt, a4paper]{article}

\author{Stefan Handschuh}
\title{research notes}

\usepackage{amssymb, amsthm, latexsym, enumerate}
\usepackage[utf8]{inputenc}
\usepackage{index}
\usepackage{geometry}
\usepackage[linktocpage=true]{hyperref}
\usepackage[T1]{fontenc}
\usepackage{algorithm}
\usepackage{algorithmic}
\usepackage{float}
\usepackage{subfigure}
\usepackage{url}
\usepackage{hyperref}
\usepackage{todonotes}

\newcommand{\BIGOP}[1]{\mathop{\mathchoice%
{\raise-0.22em\hbox{\huge $#1$}}%
{\raise-0.05em\hbox{\Large $#1$}}{\hbox{\large $#1$}}{#1}}}

\newtheorem{theorem}{Theorem}[section]
\newtheorem{remark}[theorem]{Remark}

\usepackage{graphicx}
\usepackage{tikz}
\usepackage{amsmath}
\usepackage{listings}

\newcommand{\K}{\mathbb {K}}
\newcommand{\N}{\mathbb {N}}

\title{Changing the topology of Tensor Networks}
\author{Stefan Handschuh\thanks{\href{mailto:stefan.handschuh@mis.mpg.de}
        {\nolinkurl{stefan.handschuh@mis.mpg.de}}, Max-Planck-Institute for
        Mathematics in the Sciences, Leipzig, Germany}}

\begin{document}

\maketitle

\begin{abstract}
    In many applications, it is needed to change the topology of a tensor
    network directly and without approximation. This work will introduce a
    general scheme that satisfies these needs. We will describe the procedure
    by two examples and show its efficiency in terms of memory consumption and
    speed in various numerical experiments. In general, we are going to
    provide an algorithm to add an edge to a tensor network as well as an
    algorithm to remove an edge unless the resulting network is a connected
    graph.
\end{abstract}
{\bf Keywords:} tensor format, tensor network, conversion, TT, TC, PEPS

\section{Introduction}
Tensor networks are of interest especially in quantum chemistry (see
\cite{PhysRevB.82.205105, PhysRevB.83.134421}) but also for general large data
sets, they can become of practical use.\\
Changing the topology of a tensor network, i.e. representing a tensor given in
structure $A$ as a tensor in structure $B$ may result in the opportunity to
treat originally differently structured tensors equally.

One main interest is to convert an arbitrary tensor network into a tree
structured tensor network which has nice properties in terms of stability and
computational effort. More about the indicated properties can be found in
\cite{HRS:2010}, \cite{springerlink:10.1007/s00041-009-9094-9} and
\cite{oseledets:2295}.

We want to consider only conversion from connected graph structured tensors
to connected graph structured tensors where rank $1$ edges will be neglected.

\section{Problem description}
One of the main problems with arbitrary tensor networks is, that they might
contain cycles. This leads in the representation to be not stable (see
\cite{Landsberg:arXiv1105.4449}) which results in additional conditions that
have to be preserved while performing algorithms.
With having a procedure that \textit{stabilizes} the representation by
changing its structure to a tree structure (which is stable, see
\cite[Lemma 8.6 for the Tucker format and Lemma 11.55 for the Hierarchical
format]{H:2012_1} and \cite[Theorem 3.2]{EHHS:2011}) we can avoid the
constraints. For example, one could transform the structure into a tree,
perform stable computation with it and re-transform it back to the original
structure.

Another problem is that for general tensor networks the contraction (i.e.
carrying out the summations) is hard to perform in terms of the computational
cost. If an algorithm has to compute the inner product after each iteration,
this will become a main part of the whole cost of the algorithm. Contracting
tree structured tensor networks however, has a much smaller complexity that is
linear in the dimension of the tensor network.

Furthermore, it is important for the addition of tensors that are represented
in networks, that the formats of all terms have a matching structure.
Performing computations with differently structured tensor networks is in
general harder than performing computations with equally structured ones. This
even holds for different tree structures which should be avoided. See
\cite[Section 5.2]{springerlink:10.1007/s00041-009-9094-9} for an example.

\section{Direct conversion from TC to TT w/o approximation} \label{sec:tc2tt}
Our first example will be the topology change from a ring structure (Tensor
Chain or TC, see below) to a string structure (Tensor Train or TT, see
below). The two topologies differ only in one edge and therefore
they have a lot in common which we can use for our advantage.

Let $d \in \N_{>2}, n_1, \ldots, n_d, r_1, \ldots, r_d \in \N$, then we define
    \[v = \sum_{j_1, \ldots, j_d = 1}^{r_1, \ldots, r_d} v_1(j_d, j_1) \otimes
            v_2(j_1, j_2) \otimes \ldots \otimes v_d(j_{d-1}, j_d) \in
            \bigotimes_{\mu=1}^d \K^{n_\mu} \]
with
\begin{align*}
    v_1 & : \{1, \ldots, r_d\} \times \{1, \ldots, r_1\} \rightarrow \K^{n_1}
            \\
    v_i & : \{1, \ldots, r_{i-1}\} \times \{1, \ldots, r_i\} \rightarrow
            \K^{n_i} \qquad \text{for } i = 2, \ldots, d
\end{align*}
as a \textit{Tensor Chain representation} with representation rank
$(r_1, \ldots, r_d)$, see Figure \ref{fig:tc}.

\begin{figure}[H]
    \begin{center}
        \includegraphics[width=1.00\textwidth]{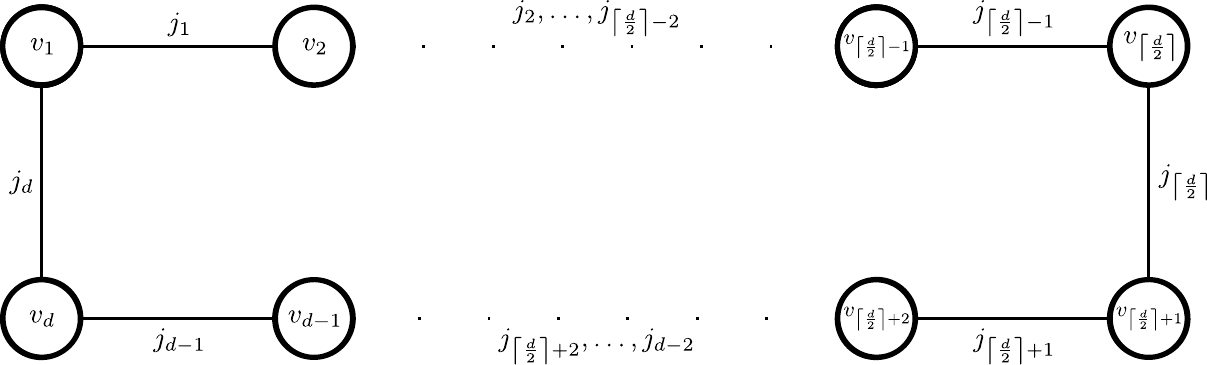}
        \caption{Tensor Chain of order $d$}
        \label{fig:tc}
    \end{center}
\end{figure}

Our goal structure is defined by
    \[ \tilde{v} = \sum_{j_1, \ldots, j_{d-1}}^ {\tilde{r}_1, \ldots,
            \tilde{r}_{d-1}} v_1(j_1) \otimes v_2(j_1, j_2) \otimes \ldots
            \otimes v_{d-1}(j_{d-2}, j_{d-1}) \otimes v_d(j_{d-1}) \]
with
\begin{align*}
    v_1 &: \{1, \ldots, \tilde{r}_1\} \rightarrow \K^{n_1}\\
    v_i &: \{1, \ldots, \tilde{r}_{i-1}\} \times \{1, \ldots, \tilde{r}_i\}
            \rightarrow \K^{n_i} \qquad \text{for } i = 2, \ldots, d-1\\
    v_d &: \{1, \ldots, \tilde{r}_{d-1}\} \rightarrow \K^{n_d}
\end{align*}
which is called a \textit{Tensor Train representation} with representation rank
$(\tilde{r}_1, \ldots, \tilde{r}_{d-1}) \in \N^{d-1}$ and visualized in Figure
\ref{fig:tt}. This is a special case of the Tensor Chain format (see
\cite{K:2010}).

\begin{figure}[H]
    \begin{center}
        \includegraphics[width=1.00\textwidth]{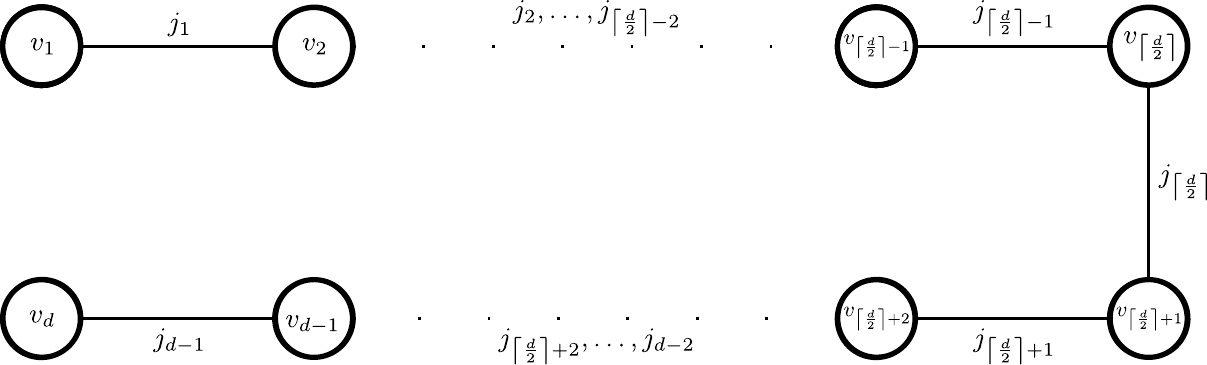}
        \caption{Tensor Train of order $d$}
        \label{fig:tt}
    \end{center}
\end{figure}

To be able to use the given ring structure of the Tensor Chain, we will convert
it to the simplest possible order $d$ tree, which is the Tensor Train. We
successively move one specific edge of the ring further and further to
the edge that is at the \textit{center} of the ring without this specific
\textit{moving} edge.
The following part visualizes this scheme.

In the following description, we are using the singular value decomposition
(SVD) to decompose a matrix. We could also utilize other decompositions, like
the QR decomposition, but they have the same computational complexity as the
SVD. A main advantage of the SVD is, that it provides a best rank $k$
approximation for matrices which we want to use later in approximated results.

\subsection*{1st step}
We define
    \[v_{1,2} (j_d, j_2) := \sum_{j_1 = 1}^{r_1} v_1(j_d, j_1) \otimes v_2
            (j_1, j_2)\]
and interpret $v_{1,2}$ as $n_1 \times n_2\cdot r_d \cdot r_2$ matrix on which
we apply the SVD to obtain
    \[ v_{1,2} (j_d, j_2) = \sum_{j_1 = 1}^{\tilde{r}_1} v'_1(j_1) \otimes
            v'_2(j_1, j_2, j_d), \]
where $\tilde{r}_1 \leq n_1$ is the full SVD rank. Consequently,
    \[v = \sum_{j_1=1}^{\tilde{r}_1} \sum_{j_2, \ldots, j_d=1}^{r_2, \ldots,
            r_d} v'_1(j_1) \otimes v'_2(j_1, j_2, j_d) \otimes v_3(j_2, j_3)
            \otimes \ldots \otimes v_d(j_{d-1}, j_d) ,\]
whose schematic representation is Figure \ref{fig:tc_step1}.
\begin{figure}[H]
    \begin{center}
        \includegraphics[width=1.00\textwidth]{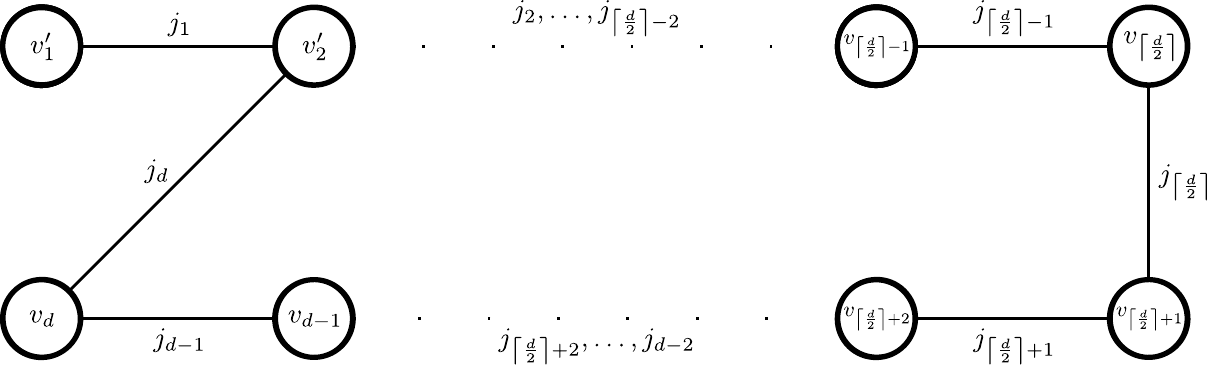}
        \caption{Structure after the 1st step}
        \label{fig:tc_step1}
    \end{center}
\end{figure}

\subsection*{2nd step}
Analogous to step 1, we define
    \[v_{d-1,d} (j_{d-2}, j_d) := \sum_{j_{d-1} = 1}^{r_{d-1}} v_{d-1}(j_{d-2},
            j_{d-1}) \otimes v_d (j_{d-1}, j_d)\]
and interpret $v_{d-1,d} (j_{d-2}, j_d)$ as $n_{d-1} \times n_d \cdot r_{d-2}
\cdot r_{d}$ matrix, of which we compute the SVD, in order to get
    \[v_{d-1,d} (j_{d-2}, j_d) = \sum_{j_{d-1}=1}^{\tilde{r}_{d-1}} v'_{d-1}
            (j_{d-2}, j_{d-1}, j_d) \otimes v'_d(j_{d-1})\]
where again $\tilde{r}_{d-1}$ is the full SVD rank. The result is
\begin{align*}
    v = \sum_{j_1, j_{d-1}=1}^{\tilde{r}_1, \tilde{r}_{d-1}} \sum_{j_2, \ldots,
            j_{d-2}, j_d=1}^{r_2, \ldots, r_{d-2}, r_d} & v'_1(j_1) \otimes
            v'_2(j_1, j_d, j_2) \otimes v_3(j_2, j_3) \otimes \ldots \otimes
            v_{d-2}(j_{d-3}, j_{d-2})\\
    & \otimes v'_{d-1}(j_{d-2}, j_{d-1}, j_d) \otimes v'_d(j_{d-1})
\end{align*}
as visualized in Figure \ref{fig:tc_step2}.

\begin{figure}[H]
    \begin{center}
        \includegraphics[width=1.00\textwidth]{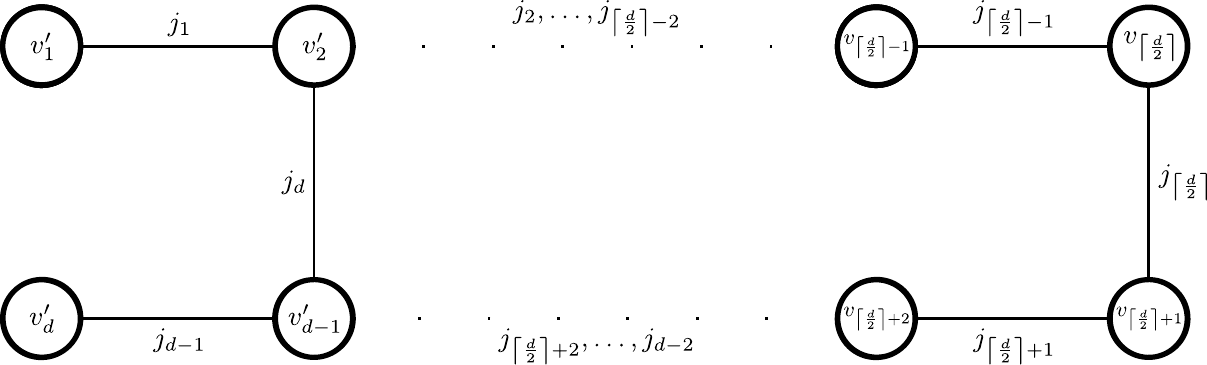}
        \caption{Structure after the 2nd step}
        \label{fig:tc_step2}
    \end{center}
\end{figure}

\subsection*{Penultimate step}
We apply the above written scheme successively, we end up in a situation that
is equivalent to Figure \ref{fig:tc_preprelaststep}.
\begin{figure}[H]
    \begin{center}
        \includegraphics[width=1.00\textwidth]{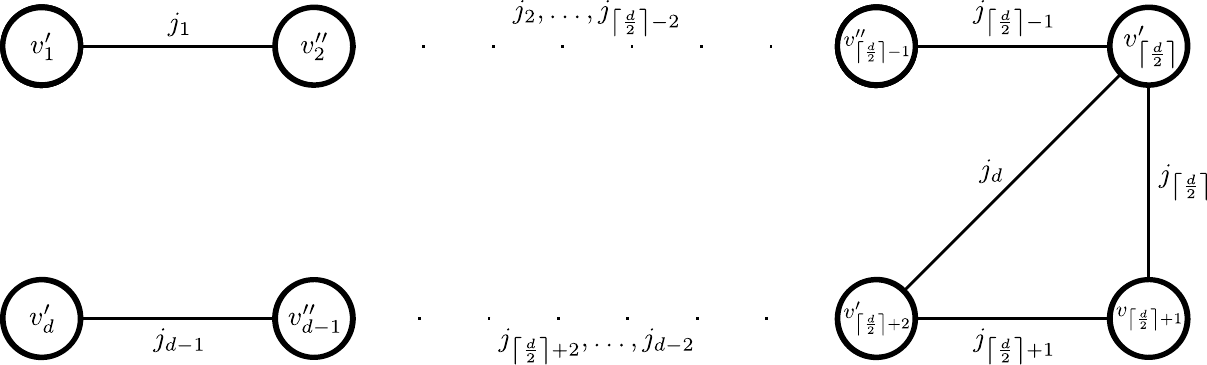}
        \caption{Structure before the penultimate step}
        \label{fig:tc_preprelaststep}
    \end{center}
\end{figure}
The penultimate step is to apply the SVD to
\begin{align*}
    \sum_{j_{\left\lceil \frac{d}{2} \right\rceil + 1} = 1}^{r_{\left\lceil
            \frac{d}{2} \right\rceil + 1}} v_{\left\lceil \frac{d}{2}
            \right\rceil+1} \left(j_{\left\lceil \frac{d}{2} \right\rceil},
            j_{\left\lceil \frac{d}{2} \right\rceil+1}\right) \otimes v'
            _{\left\lceil \frac{d}{2} \right\rceil+2} \left(j_{\left\lceil
            \frac{d}{2} \right\rceil+1}, j_{\left\lceil \frac{d}{2}
            \right\rceil+2}, j_d\right) & \stackrel{SVD}{=}\\
    \sum_{j_{\left\lceil \frac{d}{2} \right\rceil + 1} = 1}^{\tilde{r}
            _{\left\lceil \frac{d}{2} \right\rceil + 1}} v'_{\left\lceil
            \frac{d}{2} \right\rceil+1} \left(j_{\left\lceil \frac{d}{2}
            \right\rceil}, j_{\left\lceil \frac{d}{2} \right\rceil+1},
            j_d\right) \otimes v''_{\left\lceil \frac{d}{2} \right\rceil+2}
            \left(j_{\left\lceil \frac{d}{2} \right\rceil+1}, j_{\left\lceil
            \frac{d}{2} \right\rceil+2}\right)
\end{align*}
and obtain
\begin{align*}
    v = \sum_{j_1, \ldots, j_{\left\lceil \frac{d}{2} \right\rceil-1}, j
            _{\left\lceil \frac{d}{2} \right\rceil+1}, \ldots, j_{d-1} = 1}
            ^{\tilde{r}_1, \ldots, \tilde{r}_{\left\lceil \frac{d}{2}
            \right\rceil-1}, \tilde{r}_{\left\lceil \frac{d}{2}
            \right\rceil+1}, \ldots, \tilde{r}_{d-1}} \sum_{j_{\left\lceil
            \frac{d}{2} \right\rceil}, j_d = 1}^{r_{\left\lceil \frac{d}{2}
            \right\rceil}, r_d} & v'_1(j_1) \otimes v''_2(j_1, j_2) \otimes
             \ldots \otimes v''_{\left\lceil \frac{d}{2} \right\rceil-1}\left(
             j_{\left\lceil \frac{d}{2} \right\rceil-2}, j_{\left\lceil
             \frac{d}{2} \right\rceil-1}\right)\\
    & \otimes v'_{\left\lceil \frac{d}{2} \right\rceil}\left(j_{\left\lceil
            \frac{d}{2} \right\rceil-1}, j_{\left\lceil \frac{d}{2}
            \right\rceil}, j_d\right) \otimes v'_{\left\lceil \frac{d}{2}
            \right\rceil+1}\left(j_{\left\lceil \frac{d}{2} \right\rceil},
            j_{\left\lceil \frac{d}{2} \right\rceil+1}, j_d\right)\\
    & \otimes v''_{\left\lceil \frac{d}{2} \right\rceil+2}\left(j_{\left\lceil
            \frac{d}{2} \right\rceil+1}, j_{\left\lceil \frac{d}{2}
            \right\rceil+2}\right) \otimes \ldots \otimes v''_{d-1}(j_{d-2},
            j_{d-1}) \\
    & \otimes v'_d (j_{d-1})
\end{align*}
with the corresponding Figure \ref{fig:tc_prelaststep}.

\begin{figure}[H]
    \begin{center}
        \includegraphics[width=1.00\textwidth]{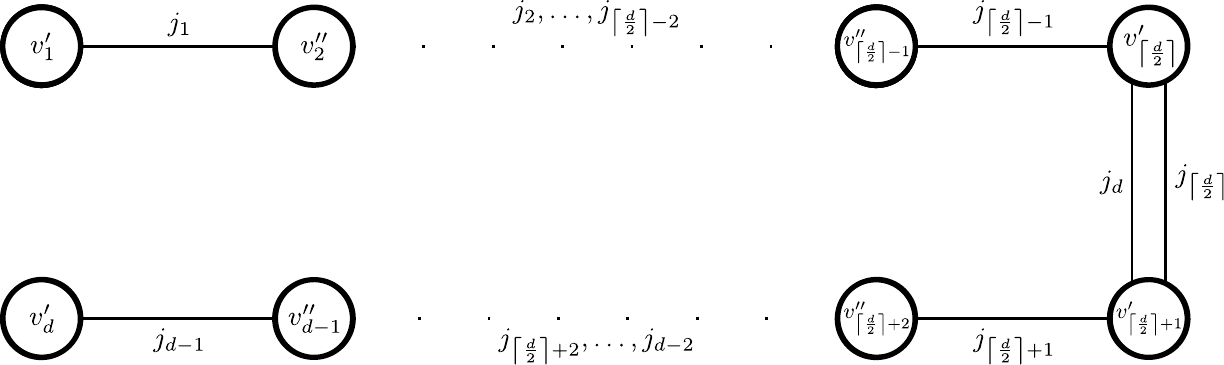}
        \caption{Structure after the penultimate step}
        \label{fig:tc_prelaststep}
    \end{center}
\end{figure}
\begin{remark} \label{rem:center_edge}
    Formally speaking, this structure is already the Tensor Train format since
    we can interpret edge $j_d$ and $j_{\left\lceil \frac{d}{2}\right\rceil}$
    as together as one edge with multiplied ranks.\\
    This situation however, can be improved by performing one additional SVD to
    combine the two edges to be able to obtain the real rank of the
    \textit{center} edge.
\end{remark}

\subsection*{Final step}
In the last step, we perform a singular value decomposition of
    \[\sum_{j_{\left\lceil \frac{d}{2} \right\rceil}, j_d=1}^{r_{\left\lceil
            \frac{d}{2} \right\rceil}, r_d} v'_{\left\lceil \frac{d}{2}
            \right\rceil}\left(j_{\left\lceil \frac{d}{2} \right\rceil-1},
            j_{\left\lceil \frac{d}{2} \right\rceil}, j_d\right) \otimes
            v'_{\left\lceil \frac{d}{2} \right\rceil+1}\left(j_{\left\lceil
            \frac{d}{2} \right\rceil}, j_{\left\lceil \frac{d}{2}
            \right\rceil+1}, j_d\right)\]
and obtain analogously to the previous step a structure that is visualized in
Figure \ref{fig:tc_laststep}. The distinction of $'$ and $''$ is important
since $''$ means that this node has been changed by two SVDs whereas $'$ stands
for one SVD. On the other hand, $\tilde{} $ indicates that a double edge has
been united.

\begin{figure}[H]
    \begin{center}
        \includegraphics[width=1.00\textwidth]{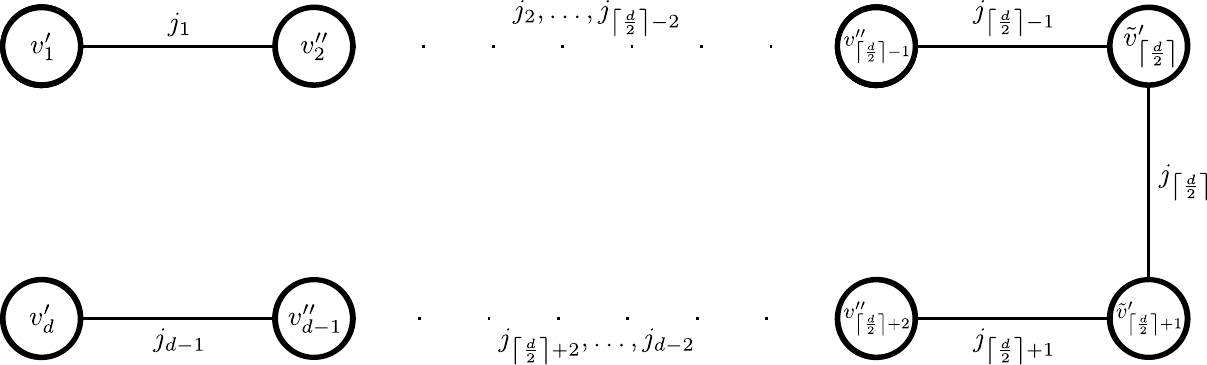}
        \caption{Completely converted structure}
        \label{fig:tc_laststep}
    \end{center}
\end{figure}

\subsection*{Ranks}
If we consider the new ranks $\tilde{r}_1, \ldots, \tilde{r}_{d-1}$, we have to
look at the dimension of the matrices that we decompose via the SVD. We have
\begin{align}
    \tilde{r}_1 &= n_1 \label{eq:tt_ranks1}\\
    \tilde{r}_i &= \min(n_i \cdot \tilde{r}_{i-1}, n_{i+1} \cdot r_d \cdot
            r_{i+1}) \leq \min(n^i, n_{i+1}\cdot r_d \cdot r_{i+1}) &
            \text{for } i = 2, \ldots, \left\lceil\frac{d}{2} \right\rceil -1
            \\
    \tilde{r}_{d-1} &= n_d\\
    \tilde{r}_i &= \min(n_{i+1} \cdot \tilde{r}_{i+1}, n_i \cdot r_d \cdot
            r_{i-1}) \leq \min(n^{d-i}, n_i \cdot r_d \cdot r_{i-1}) &
            \text{for } i=d-2, \ldots, \left \lceil \frac{d}{2} \right\rceil +1
            \label{eq:tt_ranks2}\\
    \tilde{r}_{\left \lceil \frac{d}{2} \right\rceil} &= \min(n_{\left \lceil
            \frac{d}{2} \right\rceil} \cdot \tilde{r}_{\left \lceil \frac{d}{2}
            \right\rceil-1}, n_{\left \lceil \frac{d}{2} \right\rceil+1} \cdot
            \tilde{r}_{\left \lceil \frac{d}{2} \right\rceil+1}) \leq \min(n^{
            \left\lceil \frac{d}{2} \right\rceil},n^2 \cdot r^2),
            \label{eq:tt_ranks3}
\end{align}
which is summarized in Figure \ref{fig:tt_ranks}.
\begin{figure}[H]
    \begin{center}
        \includegraphics[width=1.00\textwidth]{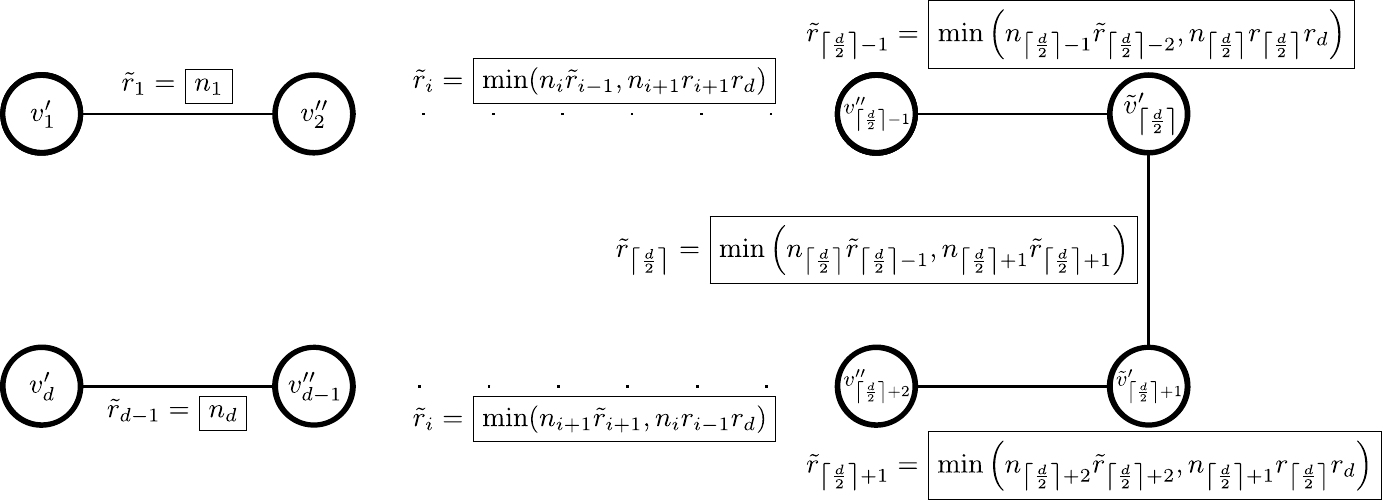}
        \caption{Final rank overview}
        \label{fig:tt_ranks}
    \end{center}
\end{figure}

\begin{remark}
    With this approach the upper bounds for the ranks are the full TT-ranks
    (see \cite{IVO:2009} and compare with \eqref{eq:tt_ranks1} --
    \eqref{eq:tt_ranks3}), but also available TC-ranks
    influence the resulting representation rank.
\end{remark}

\begin{theorem}
    The overall computational cost for the conversion is in
        \[\mathcal{O} \left((d-2) \cdot n^4 r^6 + n^6r^6\right),\]
    so it is linear in $d$ where $r := \max(r_1, \ldots, r_{d-1})$.
\end{theorem}
\begin{proof}
    Due to \eqref{eq:tt_ranks1} -- \eqref{eq:tt_ranks2}, we have
        \[\tilde{r}_i \leq n \cdot r_d \cdot r_i \leq n \cdot r^2 \qquad
                \forall i \in \{1, \ldots, d-1\} \setminus \{ \left\lceil
                \frac{d}{2} \right\rceil \}.\]
    Consequently, the matrices, that we have to decompose with the SVD have at
    most the size
        \[n \cdot \tilde{r}_i \times n \cdot r_d \cdot r \qquad \forall i \in
                \{1, \ldots, d-1\} \setminus \{ \left\lceil\frac{d}{2}
                \right\rceil \}\]
    except for the final step. There, the matrix has at most the size
        \[ n \cdot \tilde{r}_{\left\lceil \frac{d}{2} \right\rceil+1} \times n
        \cdot \tilde{r}_{\left\lceil \frac{d}{2} \right\rceil-1}\]
    due to \eqref{eq:tt_ranks3}, which finishes the proof.
\end{proof}

\begin{remark}
    Steps $1, 3, \ldots, d_1$ and $2, 4, \ldots, d_2$ are independent of each
    other and therefore parallelizable where
    \begin{align*}
        d_1 &:= \begin{cases}
            d-3 & \text{if } d \equiv 0 \mod 2,\\
            d-2 & \text{otherwise}
        \end{cases}
        \intertext{and}
        d_2 &:= \begin{cases}
            d-2 & \text{if } d \equiv 0 \mod 2,\\
            d-3 & \text{otherwise.}
        \end{cases}
    \end{align*}
\end{remark}
We can easily extend this scheme to more complex structures which we will do in
Section \ref{sec:peps2tt} by converting a $2d$ grid structured tensor into a
string structured tensor.

\subsection{Numerical example} \label{subsec:tc2tt_numerical_example}
All numerical experiments in this paper have been done with \cite{ESKWWHA:2012}
with the following setup:
\begin{itemize}
    \item the function values have are generated with a pseudo-random number
            generator
    \item each direction has $10$ entries, i.e. $n_1 = \ldots, n_d = 10$
    \item the representation rank of the tensor chain tensor is $ (r_1, \ldots,
            r_d) = (6, \ldots, 6)$
\end{itemize}
\begin{table}[H]
    \centering
    \begin{tabular}{c|c|c|c}
        $d$    & CPU-time   & Avg. rank & Max. rank \\
        \hline
        $4$    & $0.01s$    & $18.67$   & $36$  \\
        $10$   & $1557.59s$ & $188.44$  & $360$ \\
        $100$  & $1845.43s$ & $344.40$  & $360$ \\
        $1000$ & $3850.65s$ & $358.45$  & $360$
    \end{tabular}
    \caption{Exact TC to TT conversion}
    \label{tab:exacttc2tt}
\end{table}

In practice, it is often sufficient to convert a format only approximately
instead of a non-approximated conversion. Our approach can be easily changed to
an approximated conversion by using the SVD only up to a certain accuracy. We
will demonstrate this by simple computations with an allowed SVD error of
$10^{-10}$
\begin{table}[H]
    \centering
    \begin{tabular}{c|c|c|c|c}
        $d$     & CPU-time & Avg. rank & Max. rank & rel. error \\
        \hline
        $4$     & $0.01s$  & $18.67$   & $36$      & $1.49\cdot 10^{-8}$ \\
        $10$    & $2.93s$  & $30.22$   & $36$      & $1.61\cdot 10^{-7}$ \\
        $100$   & $56.6s$  & $35.47$   & $36$      & $7.12\cdot 10^{-7}$ \\
        $1000$  & $551s$   & $35.95$   & $36$      & $1.54\cdot 10^{-6}$ \\
        $10000$ & $6042s$  & $35.99$   & $36$      & $1.83\cdot 10^{-5}$
    \end{tabular}
    \caption{Approximated TC to TT conversion}
    \label{tab:approxtc2tt}
\end{table}

\begin{remark}
    We do not need to hold the whole tensor in the RAM since the conversion
    acts only locally on the two involved edges. This reduces the practical
    memory consumption to a very small fraction of the theoretical consumption
    (when storing the whole tensor in the RAM).
    Especially if we increase the accuracy of the singular value decomposition
    by increasing the rank, this \textit{locality}-advantage plays an important
    role.
\end{remark}

\section{Converting PEPS to TT w/o approximation} \label{sec:peps2tt}
In Section \ref{sec:tc2tt}, the topology changed only slightly as we removed
just one edge from the graph to obtain a tree. The method that has been used
there can be also used for more complicated structures such as $2d$ grids
which we want to explain in this section.

We are going to convert a PEPS (projected entangled pair state; see
\cite{MVC:2009} for applications) structured tensor into a tree structured
tensor in the TT-format. In our framework of arbitrary tensor representations,
a PEPS-Tensor of order $16$ has the following formula:
\begin{align*}
    v = \sum_{j_1, \ldots, j_{24} = 1}^{r_1, \ldots, r_{24}} & v_1(j_4, j_1)
            \otimes v_2(j_1, j_5, j_2) \otimes v_3(j_2, j_6, j_3) \otimes
            v_4(j_3, j_7) \otimes\\
    & v_5(j_4, j_8, j_{11}) \otimes v_6(j_8, j_5, j_{12}, j_9) \otimes
            v_7(j_9, j_6, j_{13}, j_{10}) \otimes v_8(j_{10}, j_7, j_{14})
            \otimes \\
    &v_9(j_{11}, j_{18}, j_{15}) \otimes v_{10}(j_{15}, j_{12}, j_{19},
            j_{16}) \otimes v_{11}(j_{16}, j_{13}, j_{20}, j_{17}) \otimes
            v_{12} (j_{17}, j_{14}, j_{21}) \otimes \\
    &v_{13}(j_{18}, j_{22}) \otimes v_{14}(j_{22}, j_{19}, j_{23}) \otimes
            v_{15}(j_{23}, j_{20}, j_{24}) \otimes v_{16}(j_{24}, j_{21}),
\end{align*}
see Figure \ref{fig:peps} for the visualization.
The Motivation for this conversion is due to the fact that the complexity of
contracting a PEPS tensor is very high and the optimization procedure is not
stable (see \cite{HWS:2010} for an approximated contraction scheme). Tree
structured tensors on the other hand are easy to contract and stable.

Each tree with $p$ vertices has $p-1$ edges such that it is reasonable to
choose the simplest tree structure, which is a string, as the destination
structure. This is no restriction of the method, we just chose the string
structure only for visual reasons.

For the sake of simplified notations, we set $n_1 = \ldots = n_d =:n \in \N$,
so all our vector spaces $\mathcal{V}_\mu$ have the same dimension $n$.
\begin{figure}[H]
    \begin{center}
        \includegraphics[scale=1.0]{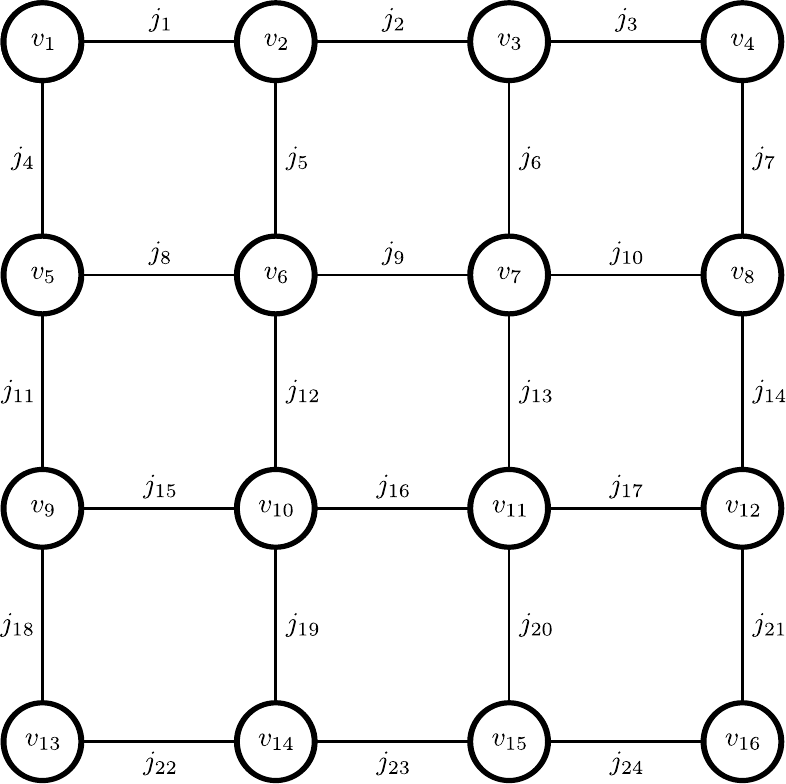}
        \caption{PEPS}
        \label{fig:peps}
    \end{center}
\end{figure}

We want to visualize the scheme that we introduced in Section \ref{sec:tc2tt}
by looking at the upper left corner of the PEPS tensor. Figure \ref{fig:part1}
displays the initial situation.\\
The first step, that we want to perform is moving the edge $j_4$ to the left.
    \[\sum_{j_1=1}^{r_1} v_1(j_4, j_1) \otimes v_2(j_1, j_2, j_5)
            \stackrel{SVD}{=} \sum_{j_1=1}^{\tilde{r}_1} v'_1(j_1) \otimes
            v'_2(j_1, j_2, j_4, j_5) \]
and we get the structure \ref{fig:part2}. Hereafter, we perform
    \[\sum_{j_8=1}^{r_8} v_5(j_4, j_8, j_{11}) \otimes v_6(j_5, j_8, j_9,
            j_{12}) \stackrel{SVD}{=} \sum_{j_8=1}^{\tilde{r}_8} v'_5(j_8,
            j_{11}) \otimes v'_6(j_4, j_5, j_8, j_9, j_{13}) \]
which is shown in Figure \ref{fig:part3}. The next step could be to move those
two edges $j_4$ and $j_5$ both further to the left, but this would increase the
complexity of the formulas as well as of the schematic drawings. Additionally,
it might be the case that the product of the moved edges ranks is to high (see
Remark \ref{rem:center_edge}).\\
So, we want to combine $j_4$ and $j_5$ into a new $j_5$ and we can do this by
one SVD:
    \[ \sum_{j_4, j_5 = 1}^{r_4, r_5} v'_2(j_1, j_2, j_4, j_5) \otimes
            v'_6(j_4, j_5, j_8, j_9, j_{12}) \stackrel{SVD}{=} \sum_{j_5=1}
            ^{\tilde{r}_5} \tilde{v}'_2(j_1, j_2, j_5) \otimes \tilde{v}'
            _6(j_5, j_8, j_9, j_{12}) \]
such that we get a structure as of Figure \ref{fig:part4}.
\begin{figure}[H]
    \begin{center}
        \subfigure[Initial state]{\label{fig:part1}
                \includegraphics[scale=1.0]{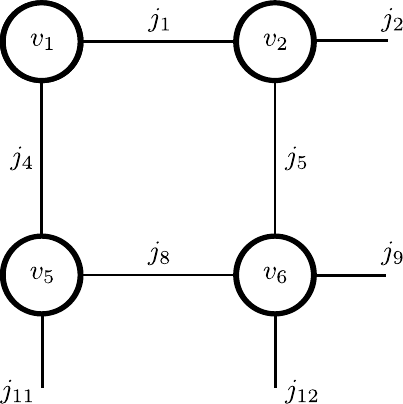}}
        \quad
        \subfigure[State after 1st step]{\label{fig:part2}
                \includegraphics[scale=1.0]{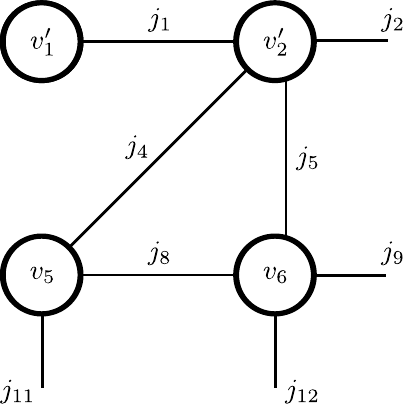}}
        \quad
        \subfigure[State after 2nd step]{\label{fig:part3}
                \includegraphics[scale=1.0]{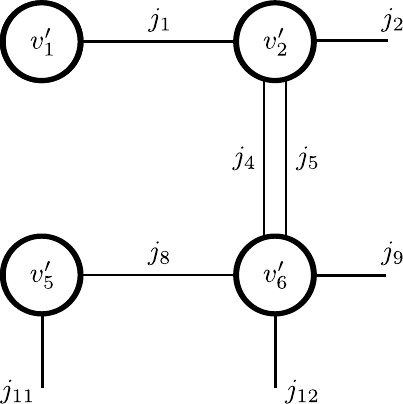}}
        \\ \hfill \\\hfill \\\hfill \\
        \subfigure[State after 3rd step]{\label{fig:part4}
                \includegraphics[scale=1.0]{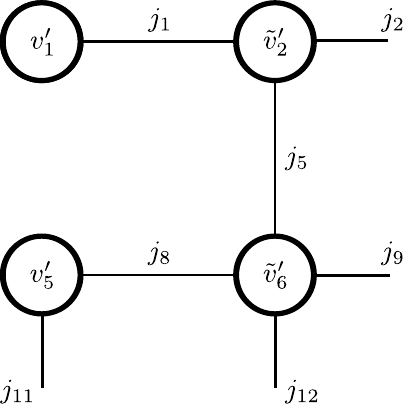}}
        \quad
        \subfigure[State after 4th step]{\label{fig:part5}
                \includegraphics[scale=1.0]{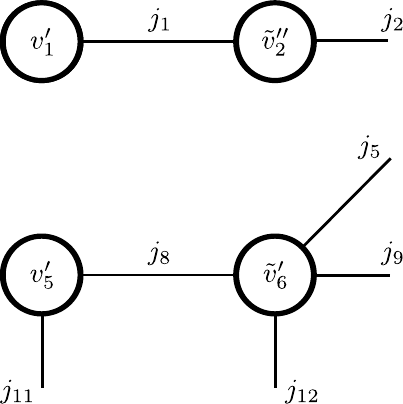}}
        \caption{Iteration series $1$ in details}
        \label{fig:parts}
    \end{center}
\end{figure}
Afterwards, we can proceed as before (see Figure \ref{fig:part5}).
If we apply this procedure until we have eliminated also edges $j_5$ and $j_6$
with edge $j_7$ left, we get a structure as in Figure \ref{fig:peps_series1}.
\begin{figure}[H]
    \begin{center}
        \includegraphics[scale=1.0]{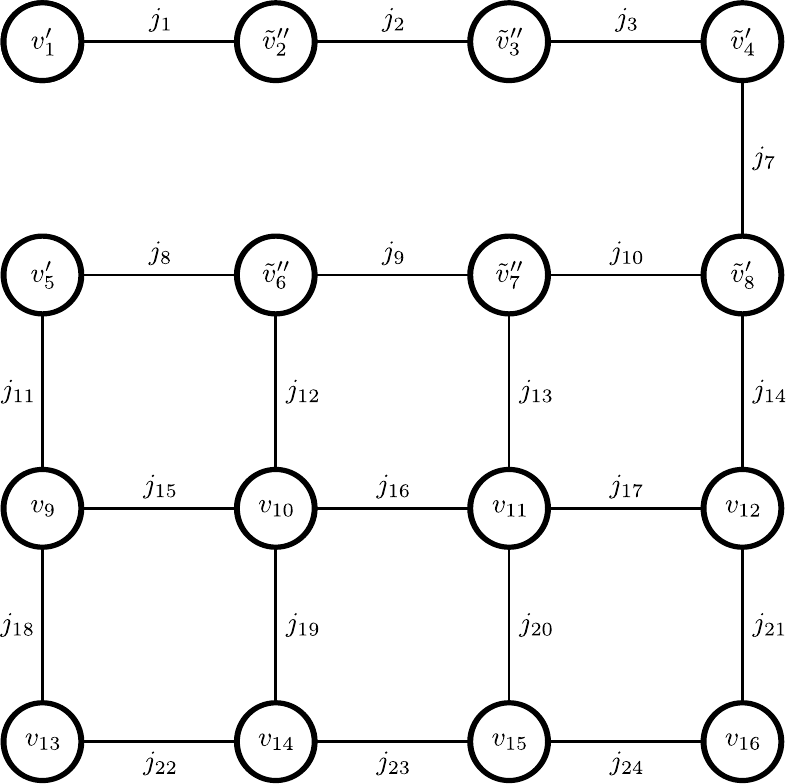}
        \caption{PEPS series 1}
        \label{fig:peps_series1}
    \end{center}
\end{figure}
Applying this scheme also on edges $(j_{18}, j_{19}, j_{20})$ we first get the
structure of Figure \ref{fig:peps_series2} and afterwards the structure of
Figure \ref{fig:peps_series3} if we eliminate edges $(j_{14}, j_{13}, j_{12})$.

\begin{figure}[H]
    \begin{center}
        \includegraphics[scale=1.0]{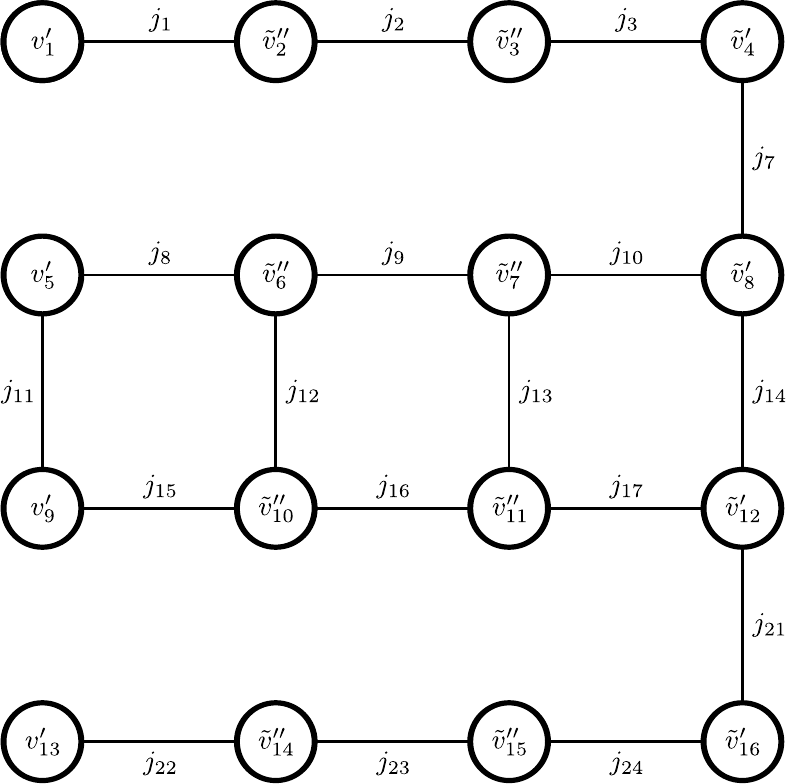}
        \caption{PEPS series 2}
        \label{fig:peps_series2}
    \end{center}
\end{figure}

The elimination processes of $(j_4, j_5, j_6)$ and of $(j_{18}, j_{19},
j_{20})$ do not affect each other such that we can state the rank distribution
independently:
\begin{align*}
    \tilde{r}_1 &= n \cdot \min(1, r_2 \cdot r_4 \cdot r_5) = n\\
    \tilde{r}_8 &= n \cdot \min(r_{11}, r_4 \cdot r_5 \cdot r_9 \cdot r_{12})\\
    \tilde{r}_5 &= n \cdot \min(\tilde{r}_1 \cdot r_2, \tilde{r}_8 \cdot r_9
            \cdot r_{12}) = n \cdot \min(n\cdot r_2, \tilde{r}_8 \cdot r_9
            \cdot r_{12})\\
    \\
    \tilde{r}_2 &= n \cdot \min(\tilde{r}_1, r_3 \cdot \tilde{r}_5 \cdot r_6)
            = n \cdot \min (n, r_3 \cdot \tilde{r}_5 \cdot r_6)\\
    \tilde{r}_9 &= n \cdot \min(\tilde{r}_8 \cdot r_{12}, \tilde{r}_5 \cdot r_6
            \cdot r_{10} \cdot r_{13})\\
    \tilde{r}_6 &= n \cdot \min(\tilde{r}_2 \cdot r_3, \tilde{r}_9 \cdot r_{10}
            \cdot r_{13})\\
    \\
    \tilde{r}_3 &= n \cdot \min(\tilde{r}_2, \tilde{r}_6 \cdot r_7)\\
    \tilde{r}_{10} &= n \cdot \min(\tilde{r}_9 \cdot r_{13}, \tilde{r}_6 \cdot
            r_7 \cdot r_{14})\\
    \tilde{r}_7 &= n \cdot \min(\tilde{r}_3, \tilde{r}_{10} \cdot r_{14}) \\
    \\
    \tilde{r}_{22} &= n \cdot \min(1, r_{18} \cdot r_{19} \cdot r_{23}) = n \\
    \tilde{r}_{15} &= n \cdot \min(r_{11}, r_{12} \cdot r_{16} \cdot r_{18}
            \cdot r_{19}) \\
    \tilde{r}_{19} &= n \cdot \min(\tilde{r}_{22} \cdot r_{23}, r_{12} \cdot
            \tilde{r}_{15} \cdot r_{16}) = n \cdot \min(n \cdot r_{23}, r_{12}
            \cdot \tilde{r}_{15} \cdot r_{16}) \\
    \\
    \tilde{r}_{23} &= n \cdot \min(\tilde{r}_{22}, \tilde{r}_{19} \cdot r_{20}
            \cdot r_{24}) = n \cdot \min(n, \tilde{r}_{19} \cdot r_{20}
            \cdot r_{24})\\
    \tilde{r}_{16} &= n \cdot \min(r_{12} \cdot \tilde{r}_{15}, r_{13} \cdot
            r_{17} \cdot \tilde{r}_{19} \cdot r_{20})\\
    \tilde{r}_{20} &= n \cdot \min(\tilde{r}_{23} \cdot r_{24}, r_{13} \cdot
            \tilde{r}_{16} \cdot r_{17})\\
    \\
    \tilde{r}_{24} &= n \cdot \min(\tilde{r}_{23}, \tilde{r}_{20} \cdot
            r_{21})\\
    \tilde{r}_{17} &= n \cdot \min(r_{13} \cdot \tilde{r}_{16}, r_{14} \cdot
            \tilde{r}_{20} \cdot r_{21})\\
    \tilde{r}_{21} &= n \cdot \min(\tilde{r}_{24}, r_{14} \cdot \tilde{r}_{17})
\end{align*}

\begin{figure}[H]
    \begin{center}
        \includegraphics[scale=1.0]{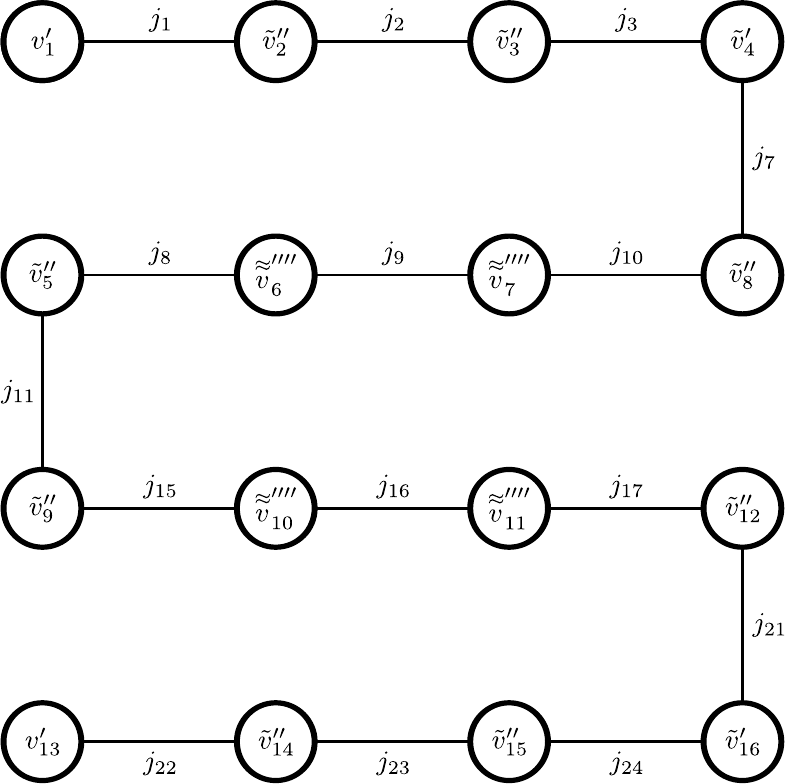}
        \caption{PEPS series 3}
        \label{fig:peps_series3}
    \end{center}
\end{figure}

\begin{align*}
    \stackrel{\approx}{r}_{10} &= n \cdot \min(\tilde{r}_7, \tilde{r}_9 \cdot
            r_{13} \cdot r_{14})\\
    \stackrel{\approx}{r}_{17} &= n \cdot \min(\tilde{r}_{21}, r_{13} \cdot
            r_{14} \cdot \tilde{r}_{16})\\
    \stackrel{\approx}{r}_{13} &= n \cdot \min(\tilde{r}_9 \cdot
            \stackrel{\approx}{r}_{10}, \tilde{r}_{16} \cdot \stackrel{
            \approx}{r}_{17})\\
    \\
    \stackrel{\approx}{r}_9 &= n \cdot \min(\stackrel{\approx}{r}_{10},
            \tilde{r}_8 \cdot r_{12} \cdot \tilde{r}_{13})\\
    \stackrel{\approx}{r}_{16} &= n \cdot \min(\stackrel{\approx}{r}_{17},
            r_{12} \cdot \tilde{r}_{13} \cdot \tilde{r}_{15})\\
    \stackrel{\approx}{r}_{12} &= n \cdot \min(\tilde{r}_8 \cdot \stackrel{
            \approx}{r}_9, \tilde{r}_{15} \cdot \stackrel{\approx}{r}_{16})\\
    \\
    \stackrel{\approx}{r}_8 &= n \cdot \min(\stackrel{\approx}{r}_9, r_{11}
            \cdot \tilde{r}_{12})\\
    \stackrel{\approx}{r}_{15} &= n \cdot \min(\stackrel{\approx}{r}_{16},
            r_{11} \cdot \tilde{r}_{12})\\
    \stackrel{\approx}{r}_{11} &= n \cdot \min(\stackrel{\approx}{r}_{15},
            \stackrel{\approx}{r}_8)
\end{align*}
Resulting in the tree structure which had to be established.

\begin{remark}
    Series $1$ and series $2$ are parallelizable without any restriction since
    they do not have a vertex in common. Series $3$ can be performed at the
    same time as series $1$ and $2$ but one has to be careful with overlapping
    cycle elimination series since it might be possible that $v_6$ is changed
    by two processes at the same time, for instance. This can be worked around
    by adding simple synchronizers. Note that there is at most one edge of the
    edges in common of two processes that may be changed simultaneously.

    Performing the conversion in parallel may lead to different ranks in the
    ranks that are adjusted more than once.
\end{remark}
\begin{remark}
    The order of the series is not unique. One can choose any other series that
    produces a string like tree.
\end{remark}

Since we are removing in general $d-2\sqrt{d}+1 = (\sqrt{d}-1)^2$ edges from
the graph, we get a complexity of the whole algorithm that is quadratic in
$d$ (with some factor $< 1$).

\section{Direct conversion from TT to TC w/o approximation}
For some applications it is needed to destroy the tree topology of a tensor
network in favor of a more complex structure. For example if one has converted
a cycle structured tensor network into a tree to perform stable algorithms and
after the computation the original structure is needed again.

The reader is reminded of the definition of the Tensor Train format that has
been introduced in Section \ref{sec:tc2tt}.

We want to convert the Tensor Train into a cyclic structured tensor
(Tensor Chain). In general, every Tensor Train is already a Tensor Chain,
since there is a rank one edge between $v_1$ and $v_d$ on every Tensor Train.
Our objective here is to get a \textit{balanced} distribution of the ranks in
the Tensor Chain and to obtain that, we have to perform a procedure that
successively moves an artificially inserted edge to the start $v_1$ and the end
$v_d$ of the train. In practice however, this leads to several problems that
are inspected in Section \ref{subsec:problems}.

This procedure depends on the singular value decomposition (SVD) and we want to
mark a node that has been change by the SVD once with $'$ and a node that has
been changed twice with $''$.

\subsection*{1st step}
Our first step will be to introduce an artificial edge between node $v_{
\left\lceil \frac{d}{2}\right\rceil}$ and $v_{\left\lceil \frac{d}{2}
\right\rceil+1}$ which we want to name $j_d$ (see Figure \ref{fig:1ststep}
for the visualization).
\begin{figure}[H]
    \begin{center}
        \includegraphics[width=1.00\textwidth]{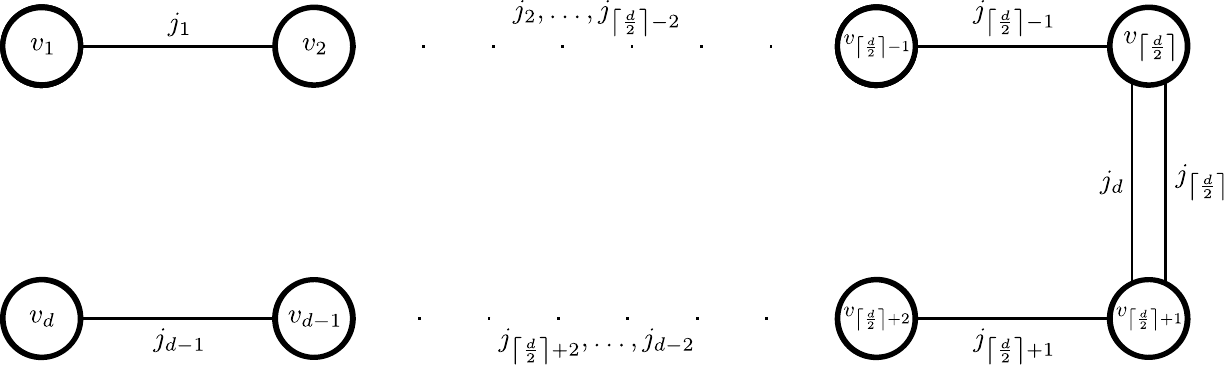}
        \caption{Artificially added edge $j_d$}
        \label{fig:1ststep}
    \end{center}
\end{figure}
We choose $r_d$ and $\tilde{r}_{\left\lceil \frac{d}{2}\right\rceil}$ such that
$r_d \cdot \tilde{r}_{\left\lceil \frac{d}{2}\right\rceil} \geq r_{\left\lceil
\frac{d}{2}\right\rceil}$ and define the mapping
\begin{align*}
    [\cdot, \cdot] : \{1, \ldots, \tilde{r}_{\left\lceil \frac{d}{2}
            \right\rceil}\} \times \{1, \ldots, r_d\} &\rightarrow \{1, \ldots,
            r_d \cdot \tilde{r}_{\left\lceil \frac{d}{2} \right\rceil}\}\\
    a,b &\mapsto a+\tilde{r}_{\left\lceil \frac{d}{2}\right\rceil} \cdot b,
\end{align*}
so $[\cdot, \cdot]$ is a bijective map to assign a $2$-tuple to a natural
number. Consequently, we have
\begin{align*}
    \sum\limits_{j_{\left\lceil \frac{d}{2}\right\rceil} = 1}^{r_{\left\lceil
            \frac{d}{2}\right\rceil}}& v_{\left\lceil \frac{d}{2}\right\rceil}
            \left(j_{\left\lceil \frac{d}{2}\right\rceil-1}, j_{\left\lceil
            \frac{d}{2}\right\rceil}\right) \otimes v_{\left\lceil \frac{d}{2}
            \right\rceil+1}\left(j_{\left\lceil \frac{d}{2}\right\rceil},
            j_{\left\lceil \frac{d}{2}\right\rceil+1}\right) = \\
    \sum\limits_{j_{\left\lceil \frac{d}{2}\right\rceil}, j_d=1}^{\tilde{r}_{
            \left\lceil \frac{d}{2}\right\rceil}, r_d}& v_{\left\lceil
            \frac{d}{2}\right\rceil} \left(j_{\left\lceil \frac{d}{2}
            \right\rceil-1}, \left[j_{\left\lceil \frac{d}{2}\right\rceil}, j_d
            \right]\right) \otimes v_{\left\lceil \frac{d}{2}\right\rceil+1}
            \left(\left[j_{\left\lceil \frac{d}{2}\right\rceil}, j_d \right],
            j_{\left\lceil \frac{d}{2}\right\rceil+1} \right).
\end{align*}

\subsection*{2nd step}
In this step, we want to move the edge $j_d$ from node $v_{\left\lceil
\frac{d}{2}\right\rceil+1}$ to $v_{\left\lceil \frac{d}{2}\right\rceil+2}$ and
as written before, we will do this with a single SVD.

\begin{figure}[H]
    \begin{center}
        \includegraphics[width=1.00\textwidth]{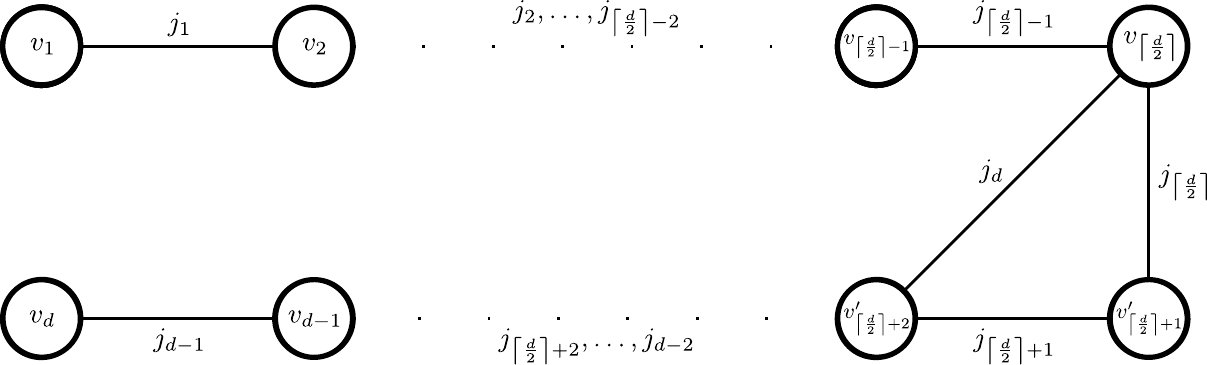}
        \label{fig:2ndstep}
    \end{center}
\end{figure}

\begin{align*}
    \sum\limits_{j_{\left\lceil \frac{d}{2}\right\rceil+1}=1}^{r_{\left\lceil
            \frac{d}{2}\right\rceil+1}}& v_{\left\lceil \frac{d}{2}\right\rceil
            +1} \left(\left[j_{\left\lceil \frac{d}{2}\right\rceil}, j_d
            \right], j_{\left\lceil \frac{d}{2}\right\rceil+1} \right) \otimes
            v_{\left\lceil \frac{d}{2}\right\rceil+2} \left(j_{\left\lceil
            \frac{d}{2}\right\rceil+1}, j_{\left\lceil \frac{d}{2}\right\rceil
            +2} \right) \stackrel{SVD}{=}\\
    \sum\limits_{j_{\left\lceil \frac{d}{2}\right\rceil+1}=1}^{\tilde{r}_{
            \left\lceil \frac{d}{2}\right\rceil+1}}&v'_{\left\lceil \frac{d}{2}
            \right\rceil+1}\left(j_{\left\lceil \frac{d}{2}\right\rceil}, j_{
            \left\lceil \frac{d}{2}\right\rceil+1} \right) \otimes v'_{
            \left\lceil \frac{d}{2}\right\rceil+2} \left(j_{\left\lceil
            \frac{d}{2}\right\rceil+1}, j_{\left\lceil \frac{d}{2}\right\rceil
            +2}, j_d\right)
\end{align*}

\subsection*{3rd step}
Edge $j_d$ has to be moved to node $v_{\left\lceil \frac{d}{2}\right\rceil-1}$
and this will be done analogously to the second step.
\begin{figure}[H]
    \begin{center}
        \includegraphics[width=1.00\textwidth]{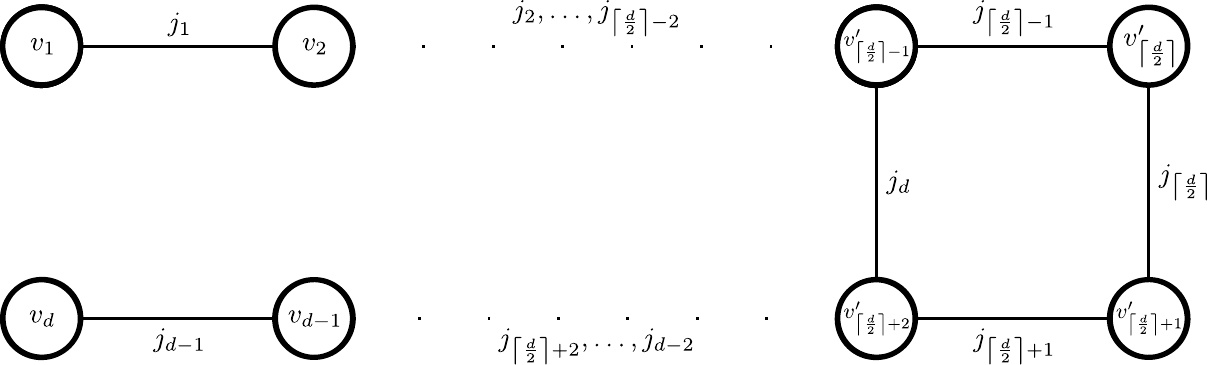}
        \caption{Situation after the 3rd step}
        \label{fig:3rdstep}
    \end{center}
\end{figure}

\begin{align*}
    \sum\limits_{j_{\left\lceil \frac{d}{2}\right\rceil-1}=1}^{r_{\left\lceil
            \frac{d}{2}\right\rceil-1}}& v_{\left\lceil \frac{d}{2}
            \right\rceil-1} \left(j_{\left\lceil \frac{d}{2}\right\rceil-2},
            j_{\left\lceil \frac{d}{2}\right\rceil-1} \right) \otimes v
            _{\left\lceil \frac{d}{2}\right\rceil} \left(j_{\left\lceil
            \frac{d}{2}\right\rceil-1}, \left[j_{\left\lceil \frac{d}{2}
            \right\rceil}, j_d \right]\right)\stackrel{SVD}{=}\\
    \sum\limits_{j_{\left\lceil \frac{d}{2}\right\rceil-1}=1}^{\tilde{r}_{
            \left\lceil \frac{d}{2}\right\rceil-1}}&v'_{\left\lceil \frac{d}{2}
            \right\rceil-1}\left(j_{\left\lceil \frac{d}{2}\right\rceil-2},
            j_{\left\lceil \frac{d}{2}\right\rceil-1}, j_d \right) \otimes
            v'_{\left\lceil \frac{d}{2}\right\rceil} \left(j_{\left\lceil
            \frac{d}{2}\right\rceil-1}, j_{\left\lceil \frac{d}{2}\right\rceil}
            \right)
\end{align*}

\begin{remark}
    Step $2$ and $3$ are independent of each other and can be performed in
    parallel.
\end{remark}

\subsection*{Final step}
After moving the edge successively further towards $v_d$ and $v_1$, we get the
situation that is visualized in Figure \ref{fig:prelaststep}.
The last step in the conversion is to move edge $j_d$ from node $v_2$ to node
$v_1$ with the described procedure.
\begin{figure}[H]
    \begin{center}
        \includegraphics[width=1.00\textwidth]{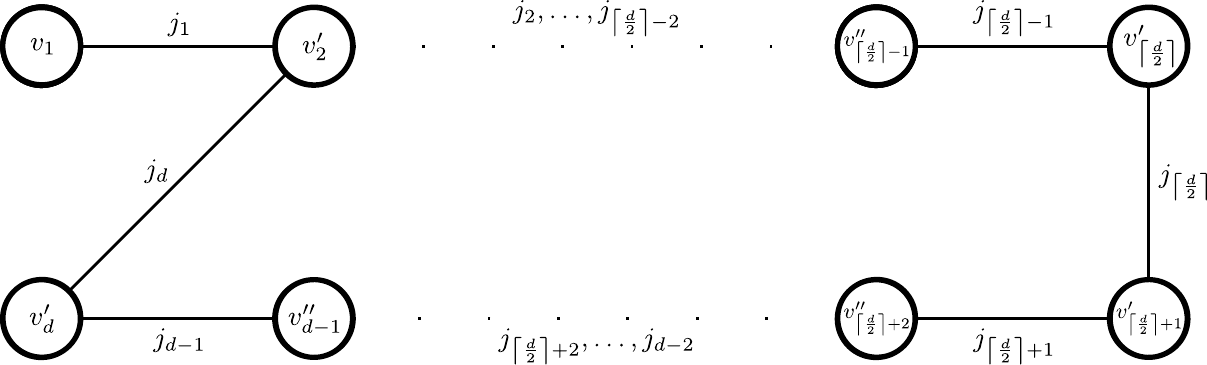}
        \caption{Situation before the final step}
        \label{fig:prelaststep}
    \end{center}
\end{figure}

So in formulas, one SVD is performed to make the edge shift:
\begin{align*}
    \left(\sum\limits_{j_1=1}^{r_1} v_1(j_1) \otimes v'_2(j_1, j_2, j_d)\right)
            _{j_d, j_2} \stackrel{SVD}{=} \left(\sum\limits_{j_1=1}^{
            \tilde{r}_1} v'_1(j_d, j_1) \otimes v''_2(j_1, j_2)\right)_{j_d,
            j_2},
\end{align*}
which is resulting in the structure that we wanted to obtain (see Figure
\ref{fig:laststep}).
\begin{figure}[H]
    \begin{center}
        \includegraphics[width=1.00\textwidth]{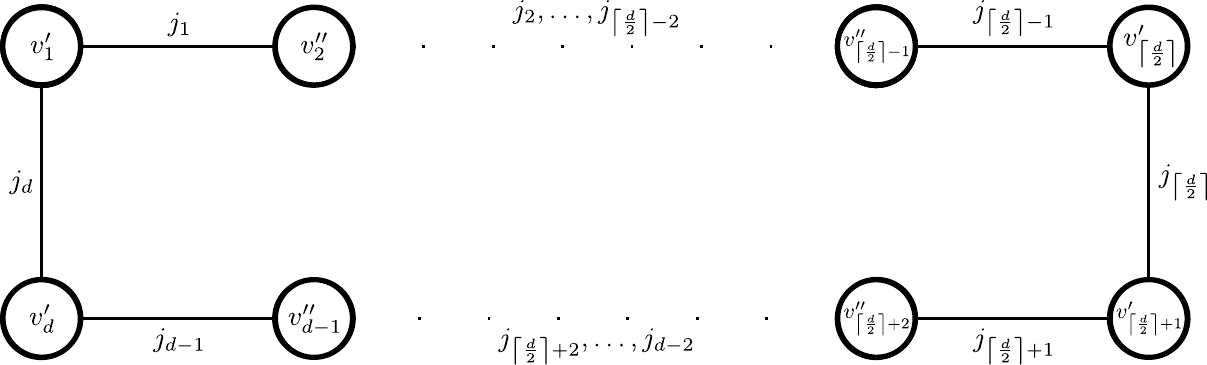}
        \caption{Situation after the final step}
        \label{fig:laststep}
    \end{center}
\end{figure}

\subsection*{Ranks}
After we have chosen the ranks $r_d$ and $\tilde{r}_{\left\lceil \frac{d}{2}
\right\rceil}$, we update all remaining $d-2$ ranks and get the following
upper bounds
\begin{align}
    \tilde{r}_i &= \min(n_i \cdot r_{i-1}\cdot r_d, n_{i+1}\cdot \tilde{r}
            _{i+1}) \leq n \cdot r \cdot r_d \qquad \text{for } i = 1, \ldots,
            \left\lceil \frac{d}{2} \right\rceil-1 \label{eq:tc_ranks1}\\
    \tilde{r}_i &= \min(n_{i+1}\cdot r_{i+1}\cdot r_d, n_i \cdot \tilde{r}
            _{i-1}) \leq n \cdot r \cdot r_d \qquad \text{for }  i =
            \left\lceil \frac{d}{2} \right\rceil +1, \ldots, d-1
            \label{eq:tc_ranks2}.
\end{align}

\begin{theorem}
    The computational cost of the described scheme is in
        \[ \mathcal{O} \left((d-2) \cdot n^4 r^3 r_d^3 \right) \]
    so it is again linear in $d$ with $r:= \max(r_1, \ldots, r_{d-1})$.
\end{theorem}
\begin{proof}
    Follows directly from Equations \eqref{eq:tc_ranks1} and
    \eqref{eq:tc_ranks2} since these equations determine the upper bound for
    the matrix sizes.
\end{proof}

\subsection{Problems} \label{subsec:problems}
The main problem has its roots in the first step where an artificial edge is
introduced into the graph structure. There we do not a priori know what the
\textit{best} rank splitting is and we also do not know which is the
\textit{best} assignment for the $[\cdot, \cdot]$ function. If we can solve
these problems, we are - for example - able to convert a tensor chain formatted
tensor into a tensor train formatted tensor and back without different ranks
for the tensor chain tensor in before the conversion and after the
back-conversion.

\subsection{Numerical example}
We have the same setup as in \ref{subsec:tc2tt_numerical_example} (except that
we are converting a Tensor Train representation into a Tensor Chain
representation) and obtain the following results with approximated SVD with a
cutoff at $10^{-10}$:
\begin{table}[H]
    \centering
    \begin{tabular}{c|c|c|c|c}
        $d$     & CPU-time & Avg. rank & Max. Rank & Rel. error \\
        \hline
        $4$     & $0.01s$  & $7.25$    & $12$      & $1.49\cdot 10^{-8}$ \\
        $10$    & $0.09s$  & $10.1$    & $12$      & $4.21\cdot 10^{-8}$ \\
        $100$   & $1.75s$  & $11.81$   & $12$      & $1.05\cdot 10^{-8}$ \\
        $1000$  & $18.5s$  & $11.98$   & $12$      & $1.17\cdot 10^{-7}$ \\
        $10000$ & $189s$   & $12$      & $12$      & $2.33\cdot 10^{-7}$
    \end{tabular}
    \caption{Approximated TT to TC conversion}
    \label{tab:approxtt2tc}
\end{table}

To illustrate the problem that has been described in Section
\ref{subsec:problems}, we will run a second experiment: first, we will
transform a tensor chain tensor into a tensor train tensor and then, we will
re-transform it back to the original chain format. In this experiment, we also
have the same setup as in \ref{subsec:tc2tt_numerical_example} (initial TC
representation rank is $(6, \ldots, 6)$). No SVD approximation is considered.
\begin{table}[H]
    \centering
    \begin{tabular}{c|c|c}
        $d$  & Avg. converted TT-rank & Avg. re-converted TC-rank\\
        \hline
        $4$  & $40$                   & $55$  \\
        $6$  & $244$                  & $224$ \\
        $8$  & $648.57$               & $1815$
    \end{tabular}
    \caption{Exact TC to TT to TC conversion}
    \label{tab:exacttc2tt2tc}
\end{table}

In the previous computation, we used the full SVD ranks such that we did not
benefit from approximated ranks. So we are going to change the algorithm to
not use the full SVD rank, but an approximated SVD rank with an accuracy of
$10^{-10}$ for each singular value decomposition for both conversions. The
error is the relative error with respect to the initial representation. The
initial TC representation rank is also $(r_1, \ldots, r_d) = (6, \ldots, 6)$.
\begin{table}[H]
    \centering
    \begin{tabular}{c|c|c|c|c}
        $d$  & Avg. converted TT-rank & Rel. error          & Avg. re-converted
                TC-rank & Rel. error\\
        \hline
        $4$  & $18.67$                & $4.47\cdot 10^{-8}$ & $33$
                        & $5.96\cdot 10^{-8}$ \\
        $6$  & $25.6$                 & $3.33\cdot 10^{-8}$ & $26$
                        & $1.86\cdot 10^{-8}$ \\
        $8$  & $28.57$                & $2.58\cdot 10^{-8}$ & $28.5$
                        & $3.65\cdot 10^{-8}$ \\
        $10$ & $30.22$                & $2.98\cdot 10^{-8}$ & $30$
                        & $1.05\cdot 10^{-8}$ \\
        $12$ & $31.27$                & $4.47\cdot 10^{-8}$ & $31$
                        & $5.58\cdot 10^{-8}$ \\
        $20$ & $33.26$                & $3.33\cdot 10^{-8}$ & $33$
                        & $4.94\cdot 10^{-8}$ \\
        $30$ & $34.21$                & $3.64\cdot 10^{-8}$ & $34$
                        & $3.65\cdot 10^{-8}$
    \end{tabular}
    \caption{Approximated TC to TT to TC conversion}
    \label{tab:approxtc2tt2tc}
\end{table}

\section{Error estimate}
While shifting an edge, we can introduce an error by omitting small singular
of the SVD's result. Doing that, we can represent matrix $A$ by an approximated
matrix $\tilde{A}$ where we can control the error $\|A-\tilde{A} \|$ with these
singular values. The influence on the whole tensor network representation has
to be investigated: We consider the change that is made in the second step of
the TC to TT conversion of Section \ref{sec:tc2tt}. We define
\begin{align*}
    v := & \sum_{j_1=1}^{\tilde{r}_1} \sum_{j_2, \ldots, j_d = 1}^{r_2,
            \ldots, r_d} v_(j_1) \otimes v_2'(j_1, j_2, j_d) \otimes v_3(j_2,
            j_3) \otimes \ldots \otimes v_d(j_{d-1}, j_d)
    \intertext{and}
    \tilde{v} := & \sum_{j_1, j_{d-1}=1}^{\tilde{r}_1, \tilde{r}_{d-1}}
            \sum_{j_2, \ldots, j_{d-2}, j_d=1}^{r_2, \ldots, r_{d-2}, r_d}
            v_1'(j_1) \otimes v_2'(j_1, j_2, j_d) \otimes v_3(j_2, j_3) \otimes
            \ldots \\
    & \qquad \otimes v_{d-2}(j_{d-3}, j_{d-2}) \otimes v_{d-1}'(j_{d-2},
            j_{d-1}, j_d) \otimes v_d'(j_{d-1}),
    \intertext{such that we have to estimate $\| v - \tilde{v} \|$. To shorten
            the notation, we introduce}
    \hat{v}(j_d, j_{d-2}) := & \sum_{j_1=1}^{\tilde{r}_1} \sum_{j_2, \ldots,
            j_{d-3} = 1}^{r_2, \ldots, r_{d-3}} v_1'(j_1) \otimes v_2'(j_1,
            j_2, j_d) \otimes v_3(j_2, j_3) \otimes \ldots \otimes
            v_{d-2}(j_{d-3}, j_{d-2})
    \intertext{and}
    B(j_d, j_{d-2}) := & \sum_{j_{d-1}=1}^{r_{d-1}} v_{d-1}(j_{d-2}, j_{d-1})
            \otimes v_d(j_{d-1}, j_d) - \sum_{j_{d-1}=1}^{\tilde{r}_{d-1}}v_{d-1}'(j_{d-2},
            j_{d-1}, j_d) \otimes v_d'(j_{d-1}).
\end{align*}
This leads into the following estimate
\begin{align*}
    \|v - \tilde{v}\| =& \left\|\sum_{j_d, j_{d-2} =1}^{r_d, r_{d-2}}
            \hat{v}(j_d, j_{d-2}) \otimes B(j_d, j_{d-2}) \right\| \\
    \leq & \sum_{j_d, j_{d-2} =1}^{r_d, r_{d-2}} \|\hat{v}(j_d, j_{d-2}) \|
            \cdot \| B(j_d, j_{d-2}) \| \\
    \leq & \left(\sum_{j_d, j_{d-2} =1}^{r_d, r_{d-2}} \|\hat{v}(j_d, j_{d-2})
            \|^2 \right)^{\frac{1}{2}} \left(\sum_{j_d, j_{d-2} =1}^{r_d,
            r_{d-2}} \|B(j_d, j_{d-2}) \|^2 \right)^{\frac{1}{2}}\\
    =&\left(\sum_{j_d, j_{d-2} =1}^{r_d, r_{d-2}} \|\hat{v}(j_d, j_{d-2})
            \|^2 \right)^{\frac{1}{2}} \cdot \|A - \tilde{A} \|
\end{align*}
with the help of the triangle inequality and the Cauchy-Schwarz-inequality (in
that order). This gives us a precise estimate on when we are allowed to cut
off singular values while still maintaining a certain error bound for $\| v -
\tilde{v} \|$.

This error estimate can be easily generalized to other tensor representations.

\section{Alternative approaches}
The proposed algorithm is of course not the only way to convert an arbitrary
tensor network into a tensor tree network. For example, one could also
evaluate the tensor network to obtain the full tensor and perform the Vidal
decomposition (see \cite{10.1137/090748330, PhysRevLett.91.147902}) in order to
obtain a tensor in the Tensor Train format. Another possibility is to decompose
the full tensor with a high order SVD (HOSVD, see \cite{lathauwer:1253}) into a
hierarchically formated tensor (see
\cite{springerlink:10.1007/s00041-009-9094-9}). Evaluating the full tensor for
large $d$ however is in general not feasible due to the amount of storage and
computational effort that is needed.

Another general approach is to fix the resulting format and use approximation
algorithms such as ALS, DMRG (both are non linear block Gauss-Seidel methods,
see \cite{EHHS:2011}). This
however is no direct conversion, but an approximation that has certain
convergence rates. The advantage there is that this approach allows us to use
general tensor representations without being restricted to tensor networks.

\bibliographystyle{unsrturl}
\bibliography{common}

\end{document}